\documentclass[12pt]{elsarticle}
\usepackage[utf8]{inputenc}
\usepackage[english]{babel}
\usepackage{amsmath}
\usepackage{amssymb}
\usepackage{amsfonts}
\usepackage{mathtools}
\usepackage{dsfont}
\usepackage{rotating}
\usepackage{graphicx}
\usepackage{floatflt,epsfig}
\usepackage{lineno,hyperref}
\usepackage{enumerate}
\usepackage{colortbl}
\usepackage{array,tabularx,tabulary,booktabs}
\usepackage{longtable}
\usepackage{multirow}
\usepackage{wrapfig}
\usepackage{subcaption}
\usepackage[table]{xcolor}
\newcolumntype{^}{>{\currentrowstyle}}

\usepackage{amsthm}
\usepackage[left=1in,right=1in,top=1in,bottom=1in]{geometry}

\usepackage{amssymb,amsthm}

\def\F{{\mathbb F}}

\newtheorem{lemma}{Lemma}[section]
\newtheorem{theorem}[lemma]{Theorem}
\newtheorem{corollary}[lemma]{Corollary}
\newtheorem{proposition}[lemma]{Proposition}
\newtheorem{remark}[lemma]{Remark}

\journal{Discrete Mathematics}
\begin{document}
\renewcommand{\abstractname}{Abstract}
\renewcommand{\refname}{References}
\renewcommand{\tablename}{Table}
\renewcommand{\arraystretch}{0.9}
\thispagestyle{empty}
\sloppy

\begin{frontmatter}
\title{Thin divisible designs graphs: an interplay between fixed-point free involutions of $(v,k,\lambda)$-graphs and symmetric weighing matrices}
\author[01]{Sergey Goryainov}
\ead{sergey.goryainov3@gmail.com}

\author[01,02]{Willem H. Haemers}
\ead{haemers@uvt.nl}

\author[03,04,05]{Elena V.~Konstantinova}
\ead{e\_konsta@math.nsc.ru}

\author[06]{Honghai Li}
\ead{lhh@jxnu.edu.cn}

\address[01] {School of Mathematical Sciences, Hebei International Joint Research Center for Mathematics and Interdisciplinary Science, Hebei Key Laboratory of Computational Mathematics and Applications, Hebei Workstation for Foreign Academicians, Hebei Normal University, Shijiazhuang  050024, P.R. China}
\address[02]{Department of Econometrics and Operations Research, Tilburg University, Tilburg, The
Netherlands}
\address[03]{Sobolev Institute of Mathematics, Ak. Koptyug av. 4, Novosibirsk 630090, Russia}
\address[04]{Three Gorges Mathematical Research Center, China Three Gorges University, 8 University Avenue, Yichang 443002, Hubei Province, P.R. China}
\address[05]{Novosibirsk State University, Pirogova Str. 2, Novosibirsk, 630090, Russia}

\address[06]{School of  Mathematics and Statistics, Jiangxi Normal University, Nanchang 330022, P.R. China}


\begin{abstract}
In this paper, we illustrate important aspects of the interplay between weighing matrices, $(v,k,\lambda)$-graphs with fixed-point free involutions, and signed graphs with an orthogonal adjacency matrix, which arises from thin divisible design graphs.
In particular, we present two new recursive constructions of regular symmetric Hadamard matrices with constant diagonal (equivalently, two new recursive constructions of strongly regular graphs) and we find a fixed-point free involution in the symplectic graph $Sp(4,q)$, where $q$ is odd, which leads to
orthogonal signings for an infinite family of antipodal distance-regular graphs of diameter 3.
\end{abstract}

\begin{keyword}
divisible design graph; thin divisible design graph; weighing matrix; fixed-point free involution; signed graphs; orthogonal signing; Hadamard matrix; symplectic graph; distance-regular graph
\vspace{\baselineskip}
\MSC[2010] 05C50\sep 05B05\sep 05E30 \sep 05B20
\end{keyword}
\end{frontmatter}

\section{Introduction} \label{sec1}
A $k$-regular graph on $v$ vertices is called a \emph{divisible design graph} (DDG) with parameters $(v,k,\lambda_1,\lambda_2,m,n)$ if its vertex set can be partitioned into $m$ classes of size $n$ such that any two distinct vertices from the same class have exactly $\lambda_1$ common neighbours and any two vertices from different classes have exactly $\lambda_2$ common neighbours. 
If $\lambda_1 = \lambda_2$ or $m = 1$ or $n = 1$, the divisible design graphs are known as $(v,k,\lambda)$-graphs (where $\lambda=\lambda_1$ or $\lambda_2$) and are also called \emph{improper} DDGs; otherwise, DDGs are called \emph{proper}.
DDGs were introduced in \cite{HKM11} as a bridge between the theory of group divisible designs and graph theory. 
Since then, it has been an active area of studies \cite{HKM11}, \cite{CH14}, \cite{GHKS19}, \cite{KS21}, \cite{S21}, \cite{K22}, \cite{CS22}, \cite{PS22}, \cite{P22}, \cite{T22} \cite{K23}, \cite{GK24}, \cite{DGHS24}, \cite{GK25}, \cite{BDG25} and \cite{K26}. 
In particular, in \cite{CH14} the extreme case of DDGs with $n = 2$ was considered; such divisible design graphs are called \emph{thin}. 

The partition from the definition of a proper DDG is clearly unique
and known to be equitable (see~\cite{HKM11}).
If for an improper DDG (that is, a $(v,k,\lambda)$-graph) the partition from the definition is equitable and $m,n>1$, 
we call the improper DDG \emph{almost proper}. 
The justification for this new term is that although the DDG is not proper, it behaves much as a proper DDG.
For proper and almost proper DDGs we will refer to the considered partition as \emph{canonical}.
It is easily seen that, for any thin DDG, proper or almost proper, the permutation swapping the vertices within each of the classes is a fixed-point free involution.
And conversely, a $(v,k,\lambda)$-graph
with a fixed-point free involution gives a thin DDG.

 The existence of an (almost) proper thin DDG of order $v$ is equivalent \cite[Section 4]{CH14} to the existence of a pair $(R,Q)$ of square matrices of order $v/2$, where $R$ is a $(0,1,2)$-matrix such that $R^2 = \alpha I + \beta J$ for some integers $\alpha$ and $\beta$ and $Q$ is a symmetric weighing matrix with diagonal entries equal to $0$ or $-1$, such that $Q \equiv R \pmod 2$. 
 This equivalence gives an interesting interplay between thin DDGs and symmetric weighing matrices (see \cite{KS97} for a background on weighing matrices). 
 Moreover, since a symmetric weighing matrix with zero main diagonal can be viewed as the orthogonal adjacency matrix of a signed graph, this interplay extends to the theory of signed graphs \cite{BCKW18}. 
 A signing of a simple graph is called \emph{orthogonal} if the signed adjacency matrix is a weighing matrix. 
 Orthogonal signings of the hypercubes were used  in the proof of the sensitivity conjecture by Huang~\cite{H19}, which makes the subject relevant and important  (note that the orthogonal signings that were used in \cite{H19} had been found earlier in \cite{AACFMN13} where signings of graphs with few distinct eigenvalues were under consideration). 
 Clearly, orthogonal adjacency matrices of signed graphs have just two distinct eigenvalues, and thus they play a role in the characterisation of signed graphs with exactly two eigenvalues, which is a research problem posed in~\cite[Problem 3.9]{BCKW18}.

In this paper, we illustrate some important aspects of the interplay described above. 
The paper is organised as follows. In Section \ref{sec:prelim}, we give preliminary definitions and results.
In Section~\ref{thin} we explain the relation between thin DDGs, symmetric weighing matrices and fixed-point free involutions of $(v,k,\lambda)$-graphs.
In Section \ref{sec:SignedCompleteMultipartiteGraphs}, we construct orthogonal signings for complete multipartite graphs using symmetric Hadamard matrices and symmetric conference matrices, and use them to construct pairs of partner matrices $(R,Q)$, which gives infinitely many thin divisible design graphs. 
As a byproduct of this process, two new recursive constructions of regular symmetric Hadamard matrices with constant diagonal occur. 
In Section \ref{sec:SymplecticGraphAsThinDDG}, we show that the complements of symplectic graphs $Sp(4,q)$, where $q$ is an odd prime power, are almost proper thin DDGs, which leads to orthogonal signings of an infinite family of antipodal distance-regular graphs of diameter 3. 

\section{Preliminaries}\label{sec:prelim}

In this section we give preliminary definitions and results.

\subsection{Distance-regular graphs}
A \emph{distance-regular graph} is a connected regular graph such that for any two vertices $x$ and $y$, the number of vertices at distance $j$ from $x$ and at distance $\ell$ from $y$ depends only upon $j$, $\ell$, and the distance between $x$ and $y$.
The \emph{intersection array} of a distance-regular graph is the array $(b_{0},b_{1},\ldots ,b_{d-1};c_{1},\ldots ,c_{d})$ in which $d$ is the diameter of the graph and for each 
$1\leq j\leq d$, $b_{j}$ gives the number of neighbours of $y$  at distance $j + 1$ from $x$, $c_{j}$ gives the number of neighbours of $y$ at distance $j-1$ from $x$ for any pair of vertices $x$ and $y$ at distance $j$. There is also the number $a_j$ that gives the number of neighbours of $y$ at distance $j$ from $x$. The numbers $a_j$, $b_j$, $c_j$ are called the \emph{intersection numbers} of the graph. They satisfy the equation $a_{j}+b_{j}+c_{j}=k$, where $k = b_0$ is the valency, that is, the number of neighbours, of any vertex.

A distance-regular graph of diameter $d$ is called \emph{antipodal} if the relation on its vertex set defined by the rule ``to coincide or to be at distance $d$'' is an equivalence relation; the equivalence classes are then called \emph{antipodal classes}. Note that whether a distance-regular graph is antipodal or not can be told \cite[Proposition 4.2.2(ii)]{BCN89} from its intersection array.

The following theorem gives an infinite family of antipodal distance-regular graphs of diameter 3; these graphs are sometimes called \emph{Mathon graphs}.

\begin{theorem}[{\cite[Proposition 12.5.3]{BCN89}}]\label{mathon}
Let $q = rm + 1$ be a prime power, where $r > 1$ and either $m$ is even or $q$ is a power of $2$. 
Let $V$ be a vector space of dimension $2$ over $\mathbb{F}_q$ provided with a nondegenerate symplectic form $B$. Let $K$ be the subgroup of the multiplicative group of $\mathbb{F}_q^*$ of index $r$, and let $b \in \mathbb{F}_q^*$. Then the graph $M(q)$ with the vertex set $\{Kx \mid x \in V\setminus \{0\}\}$ where $Kx \sim Ky$ if and only if $B(x,y) \in bK$ is an antipodal distance-regular graph with $r(q+1)$ vertices and intersection array $\{q,q-s-1,1;1,s,q\}$.   
\end{theorem}

\subsection{Strongly regular graphs}\label{SRG}
A \emph{strongly regular graph} is a regular graph with $v$ vertices and degree $k$ such that for some given integers 
$\lambda, \mu \ge 0$
every two adjacent vertices have exactly $\lambda$ common neighbours, and
every two distinct non-adjacent vertices have $\mu$ common neighbours; we then say that $(v, k, \lambda, \mu)$ are the \emph{parameters} of this strongly regular graph. The complement of a strongly regular graph with parameters $(v, k, \lambda, \mu)$ is a strongly regular graph with parameters $(v, v-k-1, v-2-2k+\mu, v-2k+\lambda)$. 
Note that the class of strongly regular graphs with $\mu > 0$ coincides with the class of distance-regular graphs of diameter 2.

A strongly regular graph with parameters $(v,k,\lambda,\mu)$ such that $\lambda = \mu$ is called a \emph{$(v,k,\lambda)$-graph}; see \cite{R71}. Note that the adjacency matrix of a $(v,k,\lambda)$-graph can be interpreted as the incidence matrix of a symmetric $(v,k,\lambda)$-design.

Let us describe an important infinite family of strongly regular graphs which are known as \emph{symplectic graphs}.  

\begin{theorem}[{\cite[Section 2.5.2]{BV22}}]
Let $t \ge 2$ be an integer and $q$ be a prime power. Let $V$ be a $2t$-dimensional vector space over $\mathbb{F}_q$. Let $B$ be a nondegenerate symplectic bilinear form on $V$. Let $Sp(2t,q)$ be the graph whose vertex set is the set of all $1$-dimensional subspaces in $V$ and two subspaces $[x]$, $[y]$ are adjacent whenever $B(x,y) = 0$, where $x$ and $y$ are vectors generating the subspaces.
The graph $Sp(2t,q)$ is strongly regular with parameters
\begin{align*}
    v &= (q^{2t}-1)/(q-1),\\
    k &= q(q^{2t-2}-1)/(q-1),\\
    \lambda &= q^2(q^{2t-4}-1)/(q-1) + q - 1,\\
    \mu &= (q^{2t-2}-1)/(q-1)
\end{align*}
so that $\lambda = \mu-2$ and $\mu = k/q$.
\end{theorem}

\begin{corollary}
The complement of the graph $Sp(2t,q)$ is a $(v,k,\lambda)$-graph with parameters
\begin{align*}
    v &= (q^{2t}-1)/(q-1),\\
    k &= q^{2t-1},\\
    \lambda &= (q-1)q^{2t-2}.
\end{align*}
\end{corollary}

\subsection{Weighing matrices, conference matrices, Hadamard matrices, and related strongly regular graphs}
A \emph{weighing matrix} of order $n$ and weight $w$ is a square matrix $W$ with entries from the set $\{0, 1, -1\}$ such that 
$WW^T = w I_n$. The rows of $W$ are pairwise orthogonal. Similarly, the columns are pairwise orthogonal. Each row and each column of $W$ has exactly $w$ non-zero elements. A weighing matrix of order $n$ and weight $w$ is denoted by $W(n,w)$. A weighing matrix $W$ is called \emph{symmetric} if $W^T = W$.

A \emph{conference matrix} is a weighing matrix of order $n$ and weight $w = n-1$. Without loss of generality one may assume that the diagonal entries of a conference matrix are zeroes. Clearly, multiplying rows and columns of a conference matrix by $-1$ results in a conference matrix with the same order and weight. A conference matrix is called \emph{normalised} if all off-diagonal entries in the first row and the first column are equal to $1$. For a normalised conference matrix $C$, the matrix $S$ obtained from $C$ by removing the first row and the first column is called the \emph{core} of $C$. 

\begin{theorem}[{\cite[Theorem 8.2.1]{BV22} \& \cite{GS67}}]
Let $C$ be a conference matrix order $n$. If $n > 1$, then $n$ is even. If $n \equiv 2\pmod 4$, then $S = S^T$. If $n \equiv 0 \pmod 4$, then $S = -S^T$.     
\end{theorem}

\begin{theorem}[{\cite[Proposition 8.2.2]{BV22}}]
Let 
$
C = \left(
\begin{array}{cc}
    0 & \textbf{1}^T \\
     \pm \textbf{1} & S
\end{array}
\right)
$  
be a conference matrix of order $n$. Then 
$
S \otimes S + I \otimes J - J\otimes I
$
is the core of a conference matrix of order $n^2+1$.
\end{theorem}

Strongly regular graphs with `half case' parameters 
$(v, k, \lambda, \mu) = (4t + 1, 2t, t -
1, t)$ are also known as conference graphs. If $S$ is the Seidel matrix of such a
graph (of order $v$), then bordering it with a first column and top row of 1's,
with 0 in the top left position, yields a symmetric conference matrix of order
$n = v + 1$, and conversely, starting with a symmetric conference matrix and
normalising yields the Seidel matrix $S$ of a strongly regular graph with `half
case' parameters.

\begin{theorem}[{\cite[Theorem 8.2.3]{BV22} \& \cite{B68} \& \cite{vLS66}}]
If $(v, k, \lambda, \mu) = (4t + 1, 2t, t-1, t)$ are the parameters of a strongly regular
graph, then $v$ is the sum of two squares.    
\end{theorem}

For example, there is no strongly regular graph with parameters $(21, 10, 4, 5)$
because 21 is not the sum of two squares. Similarly, $v = 33$ is ruled out. For all prime powers $v = 4t+1$ one has the Paley graphs (and for $v > 17$
also further examples). The smallest example of a non-prime power $v$ was given
by Mathon \cite{M78}, who constructed a family of examples including $v = 45$. An
example for the next smallest case, $v = 65$, was constructed by Gritsenko
\cite{G21}. The smallest open case is now $v = 85$. That is, it is unknown whether
there exists a symmetric conference matrix of order 86. For a survey, see \cite{BS14}.

To conclude this section, let us recall the concept of a Hadamard matrix and discuss their kinds. 
An \emph{Hadamard matrix} is a weighing matrix of order $n$ and weight $w = n$. An Hadamard matrix $H$ is called \emph{regular} when all row sums of $H$ are equal. If $J$ denotes the all-1 matrix of order $n$, then all row sums
are $a$ if and only if $HJ = aJ$. (It follows that $JH = aJ$, $a^2 = n$ and $a = \pm \sqrt{n}$.) If a regular Hadamard matrix $H$ is interpreted as the incidence matrix of a block design, with $1$ representing incidence and $-1$ representing non-incidence, then $H$ corresponds to a square $2$-$(v,k,\lambda)$ design with parameters $(4u^2, 2u^2 \pm u, u^2 \pm u)$; a design with these parameters is called a \emph{Menon design}.

The
matrix $H = (h_{ij})$ has constant diagonal when $h_{ii} = e$ for all $i$ and some fixed
$e \in \{\pm 1\}$. Abbreviate the phrase `regular symmetric Hadamard matrix with
constant diagonal' with RSHCD.

Let $H$ be a RSHCD with parameters $n, a, e$. Then $a^2 = n$ so that $a = \pm \sqrt{n}$.
The matrix $-H$ is a RSHCD with parameters $n,-a,-e$, so that there are the
two essentially distinct cases $ae > 0$ and $ae < 0$. Put 
$ae = \varepsilon \sqrt{n}$ with $\varepsilon \in \{\pm 1\}$, 
and call $H$ \emph{of type $\varepsilon$}. If $n > 1$, then $4 \mid n$, so $2 \mid a$, say $a = 2u$. Then 
$A = \frac{1}{2}(J-eH)$ is
the adjacency matrix of a $(v,k,\lambda)$-graph (complete for $(n,\varepsilon) = (4,-1)$)
with parameters
\begin{align*}
    v &= 4u^2,\\
    k &= 2u^2 - \varepsilon u,\\
    \lambda &= u^2 - \varepsilon u.
\end{align*}

Conversely, graphs with these parameters yield RSHCDs.
We see that $A$ is a symmetric incidence matrix, with zero diagonal, of a square 
$(4u^2, 2u^2\pm u; u^2\pm u)$-design
(a Menon design).

Let $T$ be the set of pairs $(n,\varepsilon)$ for which an RSHCD of order $n$ and type $\varepsilon$ exists.
In \cite[Section 8.1.1]{BV22}, the following constructions of RSHCDs were surveyed. First, there exists one recursive construction, namely, the Kronecker product
$$
(m,\delta), (n,\varepsilon) \in T \Rightarrow (mn,\delta\varepsilon) \in T. 
$$
Second, there are the following ten direct constructions:

\begin{enumerate}
    \item $(4,\pm 1), (36, \pm 1) \in T$.
    \item If there exists a Hadamard matrix of order $m$, then $(m^2, 1) \in T$; see \cite[Theorem 4.4]{GS70}.
    \item If both $a - 1$ and $a + 1$ are odd prime powers, and $4\mid a$, then $(a^2, 1) \in T$; see \cite[Theorem 4.3]{GS70}.
    \item If $a + 1$ is a prime power, and there exists a symmetric conference matrix of order $a$, then $(a^2, 1) \in T$; see \cite[Corollary 17]{SW72}.
    \item If there is a set of $t-2$ mutually orthogonal latin squares of order $2t$, then $(4t^2, 1) \in T$. 
    \item Suppose we have a Steiner system $S(2,K, V)$ with 
    $V = K(2K-1)$. If
we form the block graph, and add an isolated point, we get a graph in
the switching class of a regular two-graph. The corresponding Hadamard
matrix is symmetric with constant diagonal, but not regular. If this Steiner
system is invariant under a regular abelian group of automorphisms (which
necessarily has orbits on the blocks of sizes $V$, $V$, and $2K-1$), then by
switching with respect to a block orbit of size $V$ we obtain a strongly
regular graph with parameters
$$
v = 4K^2,~~~~k = K(2K-1),~~~~\lambda = \mu = K(K-1)
$$
showing that $(4K^2,1) \in T$. Steiner systems $S(2,K,K(2K-1))$ are known
for $K = 3, 5, 6, 7$ or $2^t$, but only for $K = 2, 3, 5, 7$ are systems known that
have a regular abelian group of automorphisms. Thus, we find $(196, 1) \in T$.
The required switching set also exists when the design is resolvable: take
the union of $K$ parallel classes. Resolvable designs are known for $K = 3$ or $2^t$; see \cite[Theorem 2.2]{BS71}.
\item $(100,-1) \in T$; \cite{JK03}.
\item If there exists a Hadamard matrix of
order $m$, then $(m^2,-1) \in T$; see \cite{H08}.
\item $(4m^4, 1) \in T$ for all positive integers $m$; see \cite{MX06}.
\item $(4m^4, -1) \in T$ for all positive integers $m$; see \cite{HX10}.
\end{enumerate}

\section{Thin divisible design graphs}\label{thin}
Let $\Gamma$ be a thin DDG with parameters $(2m,k,\lambda_1,\lambda_2,m,2)$, which is proper or almost proper.
Consider a partition of the adjacency matrix of $\Gamma$ associated with the canonical partition of the vertex set of $\Gamma$
$$
\begin{bmatrix}
A_{11} & \cdots & A_{1m}\\
\vdots & \ddots & \vdots\\
A_{m1} & \cdots & A_{mm}
\end{bmatrix},
$$
where two vertices belong to the same $(2\times2)$-block if and only if they belong to the same class of the canonical partition. 
Since $\Gamma$ is (almost) proper, each of the $(2\times 2)$-blocks takes the form
$$
A_{ij} = 
\begin{bmatrix}
a & b \\
b & a
\end{bmatrix}, \mbox{ with } a,b\in\{0,1\}.
$$

If $R$ is the quotient matrix of the canonical partition of $\Gamma$, then $(R)_{i,j} = a+b$. 
The partner $Q$ of $R$ is defined by $Q_{i,j} = a-b$.
Then $R$ is symmetric with entries $0$, $1$ or $2$, $Q$ is symmetric with entries $-1$ and $0$ or $1$.
Moreover, $R \equiv Q \pmod 2$. 
By \cite[Theorem 4.3]{CH14}, the spectrum of $\Gamma$ is the union of the spectrum of $R$ and the spectrum
of $Q$, and $Q^2 = (k-\lambda_1)I$, which means that $Q$ is a symmetric weighing matrix of weight $k-\lambda_1$ having diagonal entries equal to $-1$ or $0$.  
Conversely, the following theorem shows that a pair $(R,Q)$ with the above properties gives a thin DDG.

\begin{theorem}[{\cite[Theorem 4.4]{CH14}}]\label{thm:ThinDDGsRQ}
Let $Q$ be a symmetric weighing matrix of order $m$ and weight $w$ satisfying $(Q)_{i,i} \ne 1$ ($i = 1,\ldots,m$). Let $R$ be a symmetric $(0,1,2)$-matrix, satisfying $R \equiv Q \pmod 2$,  $(R)_{i,i} \ne 2$ ($i = 1,\ldots,m$), and $R^2 = \alpha I + \beta J$ for some integers $\alpha,\beta$, where $\beta$ is even. Then
$$
A = \frac{1}{2}
\begin{bmatrix}
R + Q & R - Q \\
R - Q & R + Q
\end{bmatrix}
$$
is the adjacency matrix of a thin DDG with quotient matrix $R$, partner $Q$, and parameters 
$$
(v = 2m, k = \sqrt{\alpha + m\beta}, \lambda_1 = k - w, \lambda_2 = \beta / 2, m, 2),
$$
where every canonical class has the form $\{i,i+m\}$ for some $i \in \{1,\ldots,m\}$.
\end{theorem}

Thus Theorem~\ref{thm:ThinDDGsRQ} gives a characterisation of thin DDGs.
Note that it is not excluded that $\lambda_1=\lambda_2$, 
so we may get improper thin DDGs, but then they are almost proper, because the partition is equitable.

Constructions of thin DDGs can be found in~\cite{CH14} and \cite{PS22}.
Let us describe two constructions from \cite{PS22} in terms of $R$ and $Q$ (we define $|Q|$ by $(|Q|)_{i,j}=|Q_{i,j}|$).

\begin{theorem}[{\cite[Constructions 20 and 21]{PS22}}]\label{thm:WeighingCompleteMultipartite1}
Let $Q$ be a symmetric ($4t, 4(t - 1))$-weighing matrix, such that the main diagonal of
$Q$ contains blocks of zeros of size $4$, and define 
$R=|Q|+2I_t\otimes(J_4-I_4)$, and $R'=|Q|+2I_{2t}\otimes(J_2-I_2)$.
Then $(Q,R)$ and $(Q,R')$ satisfy the conditions of Theorem~\ref{thm:ThinDDGsRQ}, so we obtain DDGs with parameters $(8t, 4t + 2, 6, 2t + 2, 4t, 2)$ and
$(8t, 4t - 2, 2, 2t - 2, 4t, 2)$, respectively.
\end{theorem}
\noindent
Note that the authors of \cite{PS22} mistakenly missed the word ``symmetric'' in their statement of the theorem. 
Also they showed uniqueness of the above constructions when $t=6$ and $t=8$, but did not mention how they obtained the required weighing matrices.
Since $Q$ has zero diagonal, $Q$ can be interpreted as the adjacency matrix of a
signed graph, and the underlying graph is the complete multipartite graph $K_{4,4,\ldots,4}$.
In other words, $Q$ gives an orthogonal signing of $K_{4,4,\ldots,4}$.
In the next section we present a general construction of orthogonal signings of complete multipartite graphs, leading to many new thin DDGs.

Another interesting feature of thin DDGs is the relation with fixed-point free involutions. 
The following proposition follows straightforwardly from the structure of the blocks $A_{i,j}$ given above.

\begin{proposition}[{\cite[Proposition 4.1]{CH14}}]\label{prop:InvolutionOfThinDDG}
An (almost) proper thin DDG has a fixed-point free involution, and the orbits are the classes of the canonical partition. 
\end{proposition}

In the special case that the DDG is almost proper, the DDG is a $(v,k,\lambda)$-graph. 
Then also the converse holds.

\begin{proposition}\label{involution}
If $\Gamma$ is a $(v,k,\lambda)$-graph admitting a fixed-point free involution, then with the partition into the orbits of the involution, $\Gamma$ is an almost proper thin DDG.   
\end{proposition}

\begin{proof} The orbits of the involution establish an equitable partition. 
\end{proof}

Proposition~\ref{involution} provides another approach for constructing thin DDGs, by searching for fixed point free involutions in $(v,k,\lambda)$-graphs.
For example, the Lattice graph $L(4)$ (the line graph of $K_{4,4}$) is a $(16,6,2)$-graph and
has a fixed-point free involution which interchanges only nonadjacent vertices. 
So we have an almost proper thin DDG with parameters $(16,6,2,2,8,2)$. 
This is the same DDG as the second one from Theorem~\ref{thm:WeighingCompleteMultipartite1} with $t=2$.

In Section~\ref{sec:SymplecticGraphAsThinDDG} we apply this approach and find almost proper thin DDGs from the symplectic graph $Sp(4,q)$ where $q$ is odd.

\section{Orthogonal signings of complete multipartite graphs}\label{sec:SignedCompleteMultipartiteGraphs}

In this section we construct pairs of matrices $(R,Q)$ satisfying the conditions of Theorem \ref{thm:ThinDDGsRQ}.
First, let us describe some constructions for matrices $R$.

\begin{proposition}\label{prop:MatrixR1}
Let $t \ge 2$ be an integer. Let $u \ge 1$ be an integer such that there exists an RSHCDs of order $4u^2$ and type $\varepsilon$.
Let $T$ be the adjacency block matrix of a complete $t$-partite graph with parts of size $4u^2$ with zero blocks corresponding to the parts placed along the main diagonal of $T$. Let $H_1,H_2,\ldots,H_t$ be RSHCDs (possibly equal) of order $4u^2$, type $\varepsilon$ and having $-1$'s on the main diagonal. Let $R$ be the matrix obtained from $T$ by replacing the diagonal zero blocks with the $(0,2)$-matrices $H_1 + J$, $H_2 + J$, $\ldots$, $H_t + J$, where $J$ is the all-ones matrix of order $4u^2$. Then $R$ is a symmetric $(0,1,2)$-matrix of order $4tu^2$ without $2$ on the main diagonal such that $R^2 = \alpha I + \beta J${\color{red},} where $\alpha = 4u^2$ and $\beta = 4tu^2 - 4\varepsilon u$.    
\end{proposition}
\begin{proof}
By straightforward calculations we get that the diagonal entries of $R^2$ are equal to $4u^2+ 4tu^2-4\varepsilon u$ and the non-diagonal entries of $R^2$ are equal to $4tu^2-4\varepsilon u$. The result then follows. 
\end{proof}

\begin{remark}\label{rem:MatrixRBipartite}
Note that in the special case $t= 2$ in Proposition \ref{prop:MatrixR1}, we have a matrix 
$R =
\begin{bmatrix}
H_1 + J &  J \\    
J & H_2 + J    
\end{bmatrix} 
$
with $\alpha = 4u^2$ and $\beta = 8u^2 - 4\varepsilon u$.    
\end{remark}

In Remark \ref{rem:MatrixRBipartite} we noted a matrix $R$ with $\alpha = 4u^2$ and $\beta = 8u^2 - 4\varepsilon u$. In the next proposition we give another construction of a matrix $R$ with the same parameters $\alpha$ and $\beta$.  

\begin{proposition}\label{prop:MatrixR2}
Let $u \ge 1$ be an integer. Let $H$ be a regular Hadamard matrix of order $4u^2$ and such that $\frac{1}{2}(J+H)$ is the incidence matrix of a Menon design with parameters $(4u^2,2u^2 +\delta u, u^2 + \delta u)$ where $\delta \in \{1,-1\}$ (or, equivalently, let $H$ be a regular Hadamard matrix with row sum $2\delta u$).  
Then the matrix 
$$R = 
\begin{bmatrix}
J & J+H\\
(J+H)^T & J
\end{bmatrix}
$$ is a symmetric $(0,1,2)$-matrix of order $8u^2$, without $2$ 
 on the main diagonal such that $R^2 = \alpha I + \beta J$ where $\alpha = 4u^2$ and $\beta = 8u^2 + 4\delta u$. 
\end{proposition}
\begin{proof}
Straightforward.    
\end{proof}

Further, let us describe some constructions for partner weighing matrices $Q$.

\begin{proposition}\label{prop:KroneckerProductOfHadamardAndConference}
Let $u \ge 1$ be an integer. Let $H$ be a symmetric Hadamard matrix of order $y$. Let $C$ be a symmetric conference matrix of order $t \equiv 2 \pmod 4$ with all diagonal entries equal to $0$. Then Kronecker product $C \otimes H$ is a symmetric weighing matrix of order $yt$ and weight $y(t-1)$ having $t$ zero blocks of size $y$ on the main diagonal.    
\end{proposition}

\begin{remark}
The smallest conference matrix that can be used in Proposition \ref{prop:KroneckerProductOfHadamardAndConference} is 
$
\begin{bmatrix}
0 & 1\\
1 & 0
\end{bmatrix}
$. 
\end{remark}

\begin{remark}\label{rem:OrthogonalSigningsCompleteMultipartite}
The resulting symmetric weighing matrices from Proposition \ref{prop:KroneckerProductOfHadamardAndConference} can be viewed as orthogonal signings of complete $t$-partite graphs with parts of size $y$. In particular, if $t = 2$, we get an orthogonal signing of a complete bipartite graph with parts of size $y$.
\end{remark}

The following proposition extends Remark \ref{rem:OrthogonalSigningsCompleteMultipartite} and gives a characterisation of orthogonal signings of complete bipartite graphs.

\begin{proposition}\label{prop:CharacterisationSignedCompleteBipartiteGraphs}
The existence of an orthogonal signing of a complete bipartite graph with parts of size $y$ is equivalent to the existence of an Hadamard matrix of order $y$.    
\end{proposition}
\begin{proof}
The weighing matrix that is an orthogonal signing of a complete bipartite graph with parts of size $y$ can always be written as
$\begin{bmatrix}
O & H \\
H^T & O
\end{bmatrix}$, where $H$ is an Hadamard matrix of order $y$. Conversely, if $H$ is an Hadamard matrix of order $y$, then 
$\begin{bmatrix}
O & H \\
H^T & O
\end{bmatrix}$
is a symmetric weighing matrix that is an orthogonal signing of a complete bipartite graph with parts of size $y$.
\end{proof}

The following theorem gives infinitely many pairs of matrices $(R,Q)$ that satisfy the conditions of Theorem \ref{thm:ThinDDGsRQ}.

\begin{theorem}\label{thm:ThinDDGsFromCompleteMultipartiteGraphs}
Let $t \ge 3$ be an integer. Let $u \ge 1$ be an integer. Let $H_1,H_2, \ldots, H_t$ be RSHCDs of order $4u^2$ and type $\varepsilon$. Let $R$ be the corresponding $(0,1,2)$-matrix of order $4tu^2$ obtained according to Proposition \ref{prop:MatrixR1}. Let $Q$ be a symmetric weighing matrix of order $4tu^2$ with $t$ zero blocks of size $4u^2$ on the main diagonal, that is, an orthogonal signing of a complete $t$-partite graph with parts of size $4u^2$. Then the pair $(R,Q)$ satisfies the conditions of Theorem \ref{thm:ThinDDGsRQ} and thus produces a thin DDG with parameters $(8tu^2,4tu^2-2\varepsilon u,4u^2-2\varepsilon u,2tu^2-2\varepsilon u, 4tu^2,2)$.       
\end{theorem}

\begin{remark}
Since RSHCDs of order $4$ exist for types $\varepsilon = 1$ and $\varepsilon = -1$,
Theorem \ref{thm:ThinDDGsFromCompleteMultipartiteGraphs} is a generalisation of Theorem \ref{thm:WeighingCompleteMultipartite1}. Also, in view of Proposition \ref{prop:KroneckerProductOfHadamardAndConference}, we actually get infinitely many thin DDGs in Theorem \ref{thm:ThinDDGsFromCompleteMultipartiteGraphs} while Theorem \ref{thm:WeighingCompleteMultipartite1} in its original statement in \cite{PS22} did require the existence of orthogonal signings of complete multipartite graphs with parts of size 4 but did not clarify this issue.
\end{remark}

\begin{remark}
Note that in the tuple of parameters in Theorem \ref{thm:ThinDDGsFromCompleteMultipartiteGraphs} we have $\lambda_1 = 4u^2-2\varepsilon u$, $\lambda_2 = 2tu^2-2\varepsilon u$, which implies $\lambda_1 \ne \lambda_2$ whenever $t \ge 3$, as was assumed. However, if $t = 2$, we would have equality $\lambda_1 = \lambda_2$.  
\end{remark}

Further we consider the case $t = 2$ that was excluded from Theorem \ref{thm:ThinDDGsFromCompleteMultipartiteGraphs} since, in view of Proposition \ref{prop:CharacterisationSignedCompleteBipartiteGraphs}, we can say more in this case and since we get two recursive constructions of RSHCDs deserving to be presented separately.

\begin{theorem}\label{thm:ThinAlmostProperDDGsFromCompleteBipartiteGraphs}
Let $u \ge 1$ be an integer. Let $H_1,H_2$ be RSHCDs of order $4u^2$ and type $\varepsilon$, having all diagonal entries equal to $-1$.  
Let 
$$R = 
\begin{bmatrix}
J+H_1 & J\\
J & J + H_2
\end{bmatrix}
$$ be the corresponding symmetric $(0,1,2)$-matrix of order $8u^2$ obtained according to Proposition \ref{prop:MatrixR1} with $t = 2$. 
Let 
$$Q = 
\begin{bmatrix}
O & H\\
H^T & O
\end{bmatrix}
$$ 
be a symmetric weighing matrix of order $8u^2$ and weight $4u^2$ with two zero blocks of size $4u^2$ on the main diagonal, that is, an orthogonal signing of a complete bipartite graph with parts of size $4u^2$, where $H$ is an arbitrary Hadamard matrix. Then the pair $(R,Q)$ satisfies the conditions of Theorem \ref{thm:ThinDDGsRQ} and thus produces an almost proper thin DDG with parameters 
$$(16u^2,8u^2-2\varepsilon u,4u^2-2\varepsilon u, 4u^2-2\varepsilon u, 8u^2,2)$$ 
and adjacency matrix
$$
\frac{1}{2}
\begin{bmatrix}
J + H_1 & J - H & J + H_1 & J + H \\
(J - H)^T & J + H_2 & (J + H)^T & J + H_2\\
J + H_1 & J + H & J + H_1 & J - H\\
(J + H)^T & J + H_2 & (J - H)^T & J + H_2
\end{bmatrix}.
$$       
\end{theorem}

\begin{remark}\label{rem:RecursiveConstruction1}
The almost proper thin DDGs from Theorem \ref{thm:ThinAlmostProperDDGsFromCompleteBipartiteGraphs} are actually $(v,k,\lambda)$-graphs with parameters ``RSHCD'', that is, $(v,k,\lambda)$-graphs that can be converted into RSHCDs of order $16u^2$ and type $\varepsilon$. Thus, Theorem \ref{thm:ThinAlmostProperDDGsFromCompleteBipartiteGraphs} can be interpreted as a recursive construction of RSHCDs where the input is two RSHCDs $H_1, H_2$ of order $4u^2$ and the same type $\varepsilon$ and also an arbitrary Hadamard matrix $H$ of order $4u^2$, and the output is the RSHCD 
$$
\begin{bmatrix}
H_1 & - H & H_1 & H \\
-H^T & H_2 & H^T & H_2\\
H_1 & H & H_1 & -H\\
H^T & H_2 & -H^T & H_2    
\end{bmatrix}
$$
of order $16u^2$ (that is, four times more) and the same type $\varepsilon$, or, equivalently,
$$
\begin{bmatrix}
H_1 & H_1 & -H & H \\
H_1 & H_1 & H & -H\\
-H^T & H^T & H_2 & H_2\\
H^T & -H^T & H_2 & H_2    
\end{bmatrix}.
$$
Moreover, if $H_1 = H_2 = H$, then the resulting adjacency matrix
$$
\frac{1}{2}
\begin{bmatrix}
J + H & J - H & J + H & J + H \\
J - H & J + H & J + H & J + H\\
J + H & J + H & J + H & J - H\\
J + H & J + H & J - H & J + H
\end{bmatrix}
$$
corresponds to the RSHCD 
$$
\begin{bmatrix}
H & -H & H & H \\
-H & H & H & H\\
H & H & H & -H\\
H & H & -H & H
\end{bmatrix},
$$ 
which is as the Kronecker product of the RSHCD
$$
\begin{bmatrix}
1 & -1 & 1 & 1\\
-1 & 1 & 1 & 1\\
1 & 1 & 1 & -1\\
1 & 1 & -1 & 1
\end{bmatrix}
$$
of order $4$ and positive type and the RSHCD $H$ of type $\varepsilon$.
\end{remark}

\begin{theorem}\label{thm:ThinAlmostProperDDGsFromPairOfCompleteGraphs}
Let $u \ge 1$ be an integer. Let $H$ be a regular Hadamard matrix of order $4u^2$ and with row sum $2\delta u$ where $\delta \in \{1,-1\}$.  
Let
$$R = 
\begin{bmatrix}
J & J + H\\
(J+H)^T & J
\end{bmatrix}
$$ 
be the corresponding $(0,1,2)$-matrix of order $8u^2$ obtained according to Proposition \ref{prop:MatrixR2}. 
Let $H_1,H_2$ be symmetric Hadamard matrices having all diagonal elements equal to $-1$.
Put 
$$Q = 
\begin{bmatrix}
H_1 & O\\
O & H_2
\end{bmatrix}
$$ 
to be a symmetric weighing matrix of order $8u^2$. 
Then the pair $(R,Q)$ satisfies the conditions of Theorem \ref{thm:ThinDDGsRQ} and thus produces an almost proper thin DDG with parameters 
$$(16u^2,8u^2+2\delta u,4u^2+2\delta u, 4u^2+2\delta u, 8u^2,2)$$ 
and adjacency matrix
$$
\frac{1}{2}
\begin{bmatrix}
J+H_1 & J+H & J-H_1 & J+H \\
(J+H)^T & J+H_2 & (J+H)^T & J-H_2 \\
J-H_1 & J+H  &  J+H_1 & J+H  \\
(J+H)^T & J-H_2 &  (J+H)^T & J+H_2 
\end{bmatrix}.
$$           
\end{theorem}

\begin{remark}\label{rem:RecursiveConstruction2}
The almost proper thin DDGs from Theorem \ref{thm:ThinAlmostProperDDGsFromPairOfCompleteGraphs} are actually $(v,k,\lambda)$-graphs with parameters ``RSHCD'', that is, $(v,k,\lambda)$-graphs that can be converted into RSHCDs of order $16u^2$ and type $-\delta$. Thus, Theorem \ref{thm:ThinAlmostProperDDGsFromPairOfCompleteGraphs} can be interpreted as a recursive construction of RSHCDs where the input is a regular Hadamard matrix $H$ of order $4u^2$ and with row sum $2\delta u$ where $\delta \in \{1,-1\}$ and also two symmetric Hadamard matrices $H_1, H_2$ of order $4u^2$, with all diagonal elements equal to $-1$, and the output is the RSHCD 
$$
\begin{bmatrix}
H_1 & H & -H_1 & H \\
H^T & H_2 & H^T & -H_2 \\
-H_1 & H  &  H_1 & H  \\
H^T & -H_2 &  H^T & H_2    
\end{bmatrix}
$$
of order $16u^2$ (that is, four times more) and type $-\delta$, or, equivalently,
$$
\begin{bmatrix}
H_1 & -H_1 & H & H \\
-H_1 & H_1 & H & H\\
H^T & H^T & H_2 & -H_2\\
H^T & H^T & -H_2 & H_2    
\end{bmatrix}.
$$
Moreover, if $H_1 = H_2 = H$, then the resulting adjacency matrix
$$
\frac{1}{2}
\begin{bmatrix}
J+H & J+H & J-H & J+H \\
J+H & J+H & J+H & J-H \\
J-H & J+H  &  J+H & J+H  \\
J+H & J-H &  J+H & J+H 
\end{bmatrix}
$$
corresponds to the RSHCD 
$$
\begin{bmatrix}
H & H & -H & H \\
H & H & H & -H\\
-H & H & H & H\\
H & -H & H & H
\end{bmatrix},
$$ 
which is as the Kronecker product of the RSHCD
$$
\begin{bmatrix}
1 & 1 & -1 & 1\\
1 & 1 & 1 & -1\\
-1 & 1 & 1 & 1\\
1 & -1 & 1 & 1
\end{bmatrix}
$$
of order $4$ and positive type and the RSHCD $H$ of type $-\delta$.
\end{remark}

\section{The symplectic graph}\label{sec:SymplecticGraphAsThinDDG}

In Section~\ref{SRG} we saw that the complement of the symplectic graph $Sp(2t,q)$ is a $(v,k,\lambda)$-graph.
Here we show that if $t=2$ and $q$ is odd, $Sp(4,q)$ has a fixed point free involution, which makes the complement an almost proper thin DDG.

\begin{theorem}
If $q$ is odd, the symplectic graph $Sp(4,q)$ has a fixed-point free involution, which interchanges only adjacent vertices.
\end{theorem}

\begin{proof}
Let $q$ be an odd prime power. 
For our purposes, we interpret the vector space $\F_q^4$ as a 2-dimensional space $V$ over $\F_{q^2}$.
Let $\F_q^*$ be the multiplicative group of $\F_{q}$.
Then $Sp(4,q)$ can be described as follows.
The vertex set is $\{\F_q^* v : v \in V\backslash\{0\}\}$
(this means that for $a\in\F_q^*$, $(z_1,z_2)$ and $(az_1,az_2)$ represent the same vertex).
Two vertices $(z_1, z_2)$ and $(z'_1, z'_2)$ are adjacent whenever $z_1 z'_2 - z_2 z'_1 \in\F_q$.
Write $z = x+iy \in \F_{q^2}$, with $x, y \in \F_q$ and $i^2 = c$, where $c$ is a non-square in $\F_q$, and define the map $f: V \rightarrow V$ by $f((z_1, z_2)) = (iz_1, iz_2)$.
We will show that $f$ is an involution with the required properties.
\\
(i) 
Consider the vertices $(z_1, z_2)$ and $(z'_1, z'_2)$.
Then $f((z_1,z_2))=(iz_1,iz_2)$ and $f((z'_1,z'_2))=(iz'_1,iz'_2)$
are adjacent whenever $iz_1 iz'_2 - iz_2 iz'_1 \in \F_q$.
We have $iz_1 iz'_2 - iz_2 iz'_1 = i^2(z_1 z'_2 - z_2 z'_1) =
c(z_1 z'_2 - z_2 z'_1)$, which is in $\F_q$ if and only if
$z_1 z'_2 - z_2 z'_1 \in\F_q$.
Therefore, $f$ is an automorphism.
\\
(ii)
$f(f(z_1,z_2))=f(iz_1, iz_2)=(i^2 z_1,i^2 z_2) = (cz_1,cz_2)$ with $c\in\F_q$,
hence $f$ is an involution.
\\
$(iii)$
Clearly $f(z_1, z_2) = (iz_1, iz_2) \neq (az_1,az_2)$ for every $a\in\F_q$.
So, $f$ has no fixed points.
\\
$(iv$)
The vertices $(z_1, z_2)$ and $f((z_1, z_2))=(iz_1,iz_2)$ are adjacent because $iz_1 z_2 - iz_2 z_1 = 0 \in \F_q$.
Therefore, $f$ interchanges only adjacent vertices. 
\end{proof}

\begin{corollary}
With the partition into the orbits of the involution $f$, the complement of ${Sp(4,q)}$ is an almost proper thin DDG with parameters $(v=q^3+q^2+q+1, q^3, q^2(q-1), q^2(q-1), v/2, 2)$.    
\end{corollary}

If in the above description of $Sp(4,q)$ the definition of adjacency is changed into $z_1 z'_2 - z_2 z'_1 \in\F_q^*$, then we obtain a graph $Sp(4,q)^*$ which can be obtained from $Sp(4,q)$ by deleting the edges for which $z_1 z'_2 - z_2 z'_1 = 0$;
these are the edges of a collection of $q^2+1$ disjoint cliques of order $q+1$, known as a spread (see~\cite{HT}).
The obtained graph $Sp(4,q)^*$ is an antipodal distance-regular graph of order $v$ and diameter 3, which is also a Mathon graph over the field $\F_{q^2}$ as described in Theorem~\ref{mathon}.
This description uses a subgroup $K$ of the multiplicative group of $\F_{q^2}$ of index $r$. 
If we take $K$ equal to $\F_q^*$, and $r=q+1$ we get $Sp(4,q)^*$, but if we take $K=\F_q^*\cup i\F_q^*$,
and $r=(q+1)/2$ then we get an antipodal distance-regular Mathon graph $M(q^2)$ of order $v/2$ and diameter 3 for which the adjacency matrix $B$ is the quotient matrix of $Sp(4,q)^*$ with respect to the partition of the involution.
It is not difficult to see that $B\equiv R\equiv Q$~(mod~2), where $R$ and $Q$ are the quotient matrix $R$ and its partner $Q$ of the thin DDG obtained above.
Thus, $Q$ is a weighing matrix which gives an orthogonal signing of $M(q^2)$.

\section*{Acknowledgments}
Sergey Goryainov thanks Jiangxi Normal University and Eindhoven University of Technology for hosting his visits to Nanchang and Eindhoven. Willem Haemers thanks Hebei Normal University for hosting his visit in Shijiazhuang. Elena V.~Konstantinova thanks Jiangxi Normal University and Hebei Normal University hosted her visits to Nanchang and Shijiazhuang where this project was discussed. The work of Sergey Goryainov is supported by Natural Science Foundation of Hebei Province (A2023205045). The work of Elena V.~Konstantinova is supported by the state contract of the Sobolev Institute of Mathematics, project No.~FWNF-2026-0011 ``Algebraic and combinatorial invariants of discrete structures''.  The work of Honghai Li is supported by the National Natural Science Foundation of China (Nos. 12561060, 12161047) and  Jiangxi Provincial Natural Science foundation (No. 20224BCD41001).

\end{document}